\documentclass[11pt,a4paper]{article}

\usepackage[utf8]{inputenc}
\usepackage{amsmath}
\usepackage{amssymb}
\usepackage{amsthm}
\usepackage{hyperref}

\usepackage{orcidlink}
\usepackage{authblk}

\renewcommand{\labelenumi}{\theenumi}
\renewcommand{\theenumi}{{\it(\roman{enumi})}}
\newcommand*\fixitem {\item[]%
  \refstepcounter{enumi}\hskip-\leftmargin\labelenumi\hskip\labelsep}

\newcommand{\T}{\mathbb{T}}
\newcommand{\Z}{\mathbb{Z}}
\newcommand{\C}{\mathbb{C}}
\newcommand{\N}{\mathbb{N}}

\def\zdSet{\mathcal{Z}}
\def\zhdSet{\zdSet_h}

\def\sub{\subseteq}

\def\into{\longrightarrow}
\def\setdiff{\backslash}

\def\mlim{{\displaystyle\lim_{m\rightarrow\infty}}}
\def\kliminf{{\displaystyle\liminf_{k\rightarrow\infty}}}

\def\nulimsup{{\displaystyle\limsup_{\nu\rightarrow\infty}}}

 \newtheorem{theorem}{Theorem}[section]
 \newtheorem{lemma}[theorem]{Lemma}
  \newtheorem{proposition}[theorem]{Proposition}

\title{Zero Harmonic Density implies Zero Density}
\author{Rafael Reno S. Cantuba\thanks{Associate Professor, Department of Mathematics and Statistics, De La Salle University (DLSU), Taft Ave., Manila, Philippines, supported by a grant from the Research Grants Management Office of DLSU, grant no.: 02FR1TAY24-1TAY25, email: rafael.cantuba@dlsu.edu.ph}\\ ORCID: 0000-0002-4685-8761\orcidlink{0000-0002-4685-8761}}
\date{}

\begin{document}

\maketitle

\begin{abstract}
A set of integers has zero harmonic density if its elements, viewed as frequencies of continuous characters of the unit circle as a compact group, are measure-theoretically sparse enough that the values of Fourier transforms of Radon measures supported on small open arcs can interpolate those of any global Radon measure. In contrast, zero density quantifies the proportional size of a set of integers relative to expanding intervals on the real line. Although these notions arise from abstract harmonic analysis and elementary number theory, respectively, zero harmonic density is proven as a sufficient condition for zero density, resolving what has been stated in the literature as an open problem concerning the relation between these two notions. The proof makes use of classical gap theorems from complex analysis.
\end{abstract}

\vspace{0.5em}
\noindent
\begin{quote}
    
    \vspace{0.5em}
\end{quote}

\section{Introduction}

One facet of the study of boundary behavior for complex power series is rooted in the interplay between the gaps among the integer exponents and the impossibility of analytic continuation across the circle of convergence. The former is often called ``lacunarity'' from the Latin \emph{lacuna} meaning gap.  Historically, this study began with classical examples, such as Weierstrass's continuous nowhere-differentiable series and Fredholm's example of a series with smooth boundary values that nevertheless possesses the unit circle as a natural boundary. See, for instance, \cite[Section~6.9]{ros02}. These early observations culminated in the so-called ``gap theorems'' of complex analysis, most prominently the Hadamard Gap Theorem\textemdash which establishes a natural boundary for power series whose exponent sequences grow at a multiplicative rate\textemdash and the Fabry Gap Theorem, which significantly relaxes this condition to require only that the sequence of exponents has zero density. The Fabry condition represents the sharpest possible purely exponent-based  criterion for a natural boundary, as evidenced by classical converse results due to Pólya \cite{pol42} and Erdős \cite{erd45}. These function-theoretic developments were shown to be related to the measure-theoretic notion, in abstract harmonic analysis, of thin sets, which are also widely known as interpolation sets\footnote{We are mainly using \cite{gra13} for it is a modern and comprehensive reference, but such sets appeared in earlier reference texts \cite{hew94,hew70,lop75}. Since \cite{hew94,hew70} are voluminous, we specify here: \cite[Section~26]{hew94} and \cite[Sections~37,41]{hew70}.} \cite[pp.~xi--xiii]{gra13}. These sets—including Sidon sets, $I_0$ sets, and sets of zero harmonic density—are characterized by their ability to interpolate sequences or Fourier transforms using localized Radon measures on compact abelian groups. The concept of zero harmonic density, introduced by Déchamps-Gondim\footnote{We refer the reader to the comprehensive review in \cite[Section~10.5]{gra13}.}, isolates sets of frequencies $E \subseteq \mathbb{Z}$ whose restricted Fourier transforms can mimic any global Radon measure via measures supported on arbitrarily small open subsets of the unit circle $\T$. This strong measure-theoretic sparseness was used in \cite[Chapter~10]{gra13} to prove a generalized Hadamard-type gap theorem: power series whose exponents form a set of zero harmonic density cannot be analytically continued across their circle of convergence.

Despite the seemingly central role played by zero harmonic density in generalizing classical gap theorems, its exact relation to standard asymptotic density is said to be unresolved. While it is known that zero density does not imply zero harmonic density\textemdash as illustrated by subsets of $\mathbb{Z}$ containing arbitrarily long arithmetic progressions of fixed step length\textemdash the converse direction was presented as an open question \cite[p.~186]{gra13}. Specifically, we find in \cite[p.~249]{gra13} the explicitly posed problem of determining when a set of zero density has zero harmonic density.

In this paper, we resolve this question by establishing that the class of sets of integers with zero harmonic density is strictly contained within the class of those with zero density. Our proof combines the generalized gap theorem \cite[Proposition~10.4.7]{gra13}, the converse Fabry gap theorem of Pólya and Erdős, and some Laurent series decomposition to reach a contradiction whenever a set of zero harmonic density fails to have zero density.

\section{Preliminaries}\label{PrelimSec}

We denote by $\T$ the compact group of all complex numbers with modulus one. The set of all Radon measures concentrated on $U\sub \T$ (meaning Radon measures $\nu$ such that $\nu(\T\setdiff U)=0$) shall be denoted by $M(U)$. By Pontryagin duality in the theory of topological groups from abstract harmonic analysis, the dual group of $\T$ is the discrete group $\Z$ of all the integers. In particular, under some topological group isomorphism, any integer $n$ may be identified with a character of $\T$, which is the function $e^{it}\mapsto e^{int}$. The \emph{Fourier-Stieltjes transform} of $\nu\in M(\T)$ is the function $\widehat{\nu}:\Z\into\C$ given by $\widehat{\nu}(\gamma):=\int_\T\overline{\gamma(t)}\,d\nu(t)$, where the mapping $\C\into\C$ denoted by $z\mapsto\overline{z}$ means complex conjugation. We shall also use the notation $\N:=\Z\cap(0,\infty)$, and throughout this paper, for any non-empty $S \subseteq \mathbb{Z}$, expressions of the form $\sum_{\gamma \in S} a_\gamma z^\gamma$ are understood to be summed according to the unique strictly increasing ordering of $S$.

A subset $E$ of $\Z$ is said to have \emph{zero harmonic density} if for each open subset $U$ of $\T$ and each measure $\varphi\in M(\T)$, there exists a measure $\nu\in M(U)$ such that $\gamma\in E$ implies $\widehat{\nu}(\gamma)=\widehat{\varphi}(\gamma)$. A set $E\sub\Z$ has \emph{zero density} if ${\displaystyle\limsup_{n\rightarrow\infty}}\frac{|E\cap[-n,n]|}{2n+1}=0$. Also, for every $E\sub\Z$, we define $E^+:=E\cap(0,\infty)$ and $E^-:=\{-n\  :\  n\in E\cap(-\infty,0]\}$.

\begin{proposition}\label{subProp}\begin{enumerate}\fixitem\label{zhdSub} If $E$ has zero harmonic density, then so does any $F\sub E$.
\item\label{zdSub} If $E$ does not have zero density, then either $E^+$ or $E^-$ does not have zero density. In such a case, if the strictly increasing sequence $(n_k)$ is an enumeration of whichever of $E^+$ or $E^-$ does not have zero density, then $\kliminf\frac{n_k}{k}<\infty$. 
\end{enumerate}
\end{proposition}
\begin{proof}[Proof (Sketch)] 
The statement \ref{zhdSub} follows immediately from the definition of zero harmonic density. To prove \ref{zdSub}, we use the fact that \[|E \cap [-n, n]| \le |E^+ \cap [-n, n]| + |E^- \cap [-n, n]|,\] which routinely implies
\[
\limsup_{n\to\infty} \frac{|E \cap [-n, n]|}{2n+1} \le \limsup_{n\to\infty} \frac{|E^+ \cap [-n, n]|}{2n+1} + \limsup_{n\to\infty} \frac{|E^- \cap [-n, n]|}{2n+1},
\]
the left-hand side of which is positive when $E$ does not have zero density, and this forces one of the nonnegative quantities $\limsup_{n\to\infty} \frac{|E^+ \cap [-n, n]|}{2n+1}$ and $\limsup_{n\to\infty} \frac{|E^- \cap [-n, n]|}{2n+1}$ to be also positive. Thus, $E^+$ or $E^-$ does not have zero density. 

Suppose that the strictly increasing sequence $(n_k)$ is an enumeration of $S\in\{E^+,E^-\}$ with $L:=\nulimsup s_\nu>0$, where $s_\nu:=\frac{|S\cap[-\nu,\nu]|}{2\nu+1}$ for all $\nu\in\N$. Since $S\sub[0,\infty)$, we have $S\cap[-\nu,\nu]=S\cap[0,\nu]$, so $|S\cap[-\nu,\nu]|$ cannot exceed the $\nu+1$ integers in the interval $[0,\nu]$. By a routine argument, $0\leq s_\nu\leq 1$ for all $\nu\in\N$. This means that we further have $0<L<\infty$. Consequently, $\frac{1}{L}<\infty$, which we will need later.

 By routine use of the notion of suprema, a strictly increasing sequence $(\nu_m)$ in $\N$ may be produced by induction such that $m\in\N$ implies $s_{\nu_m}>\frac{L}{2}$, which, by the definition of $s_\nu$, further implies $K_m:=|S\cap[-\nu_m,\nu_m]|>L\nu_m+\frac{L}{2}>L\nu_m$. That is, $S$ has $K_m$ elements in $[-\nu_m,\nu_m]$, and since $S\sub[0,\infty)$, we further deduce that $S$ has $K_m$ elements in $[0,\nu_m]$. Since $(n_k)$ is a strictly increasing enumeration of $S$, the $K_m$th element of $S$, which is $n_{K_m}$ must be in $[0,\nu_m]$. Consequently, $n_{K_m}\leq \nu_m$ for all $m\in\N$. By routine manipulations, we have, at this point, $\frac{n_{K_m}}{K_m}\leq \frac{\nu_m}{K_m}<\frac{1}{L}$ for all $m\in\N$. Because $(\nu_m)$ is a strictly increasing sequence in $\N$, and because $K_m>L\nu_m$ (with $L>0$) for all $m$, we find that $\mlim K_m=\infty$, so for each $K\in\N$, we have $K_m\geq K$ for sufficiently large $m$, and for such indices $m$, we also have $\frac{n_{K_m}}{K_m}\in\left\{\frac{n_k}{k}\,:\,k\geq K\right\}$, so $\inf_{k\geq K}\frac{n_k}{k}\leq \frac{n_{K_m}}{K_m}<\frac{1}{L}$. That is, $\frac{1}{L}$ is an upper bound of $\left\{\inf_{k\geq K}\frac{n_k}{k}\,:\,K\in\N\right\}$, and so, \[\kliminf\frac{n_k}{k}=\sup_{K\in\N}\inf_{k\geq K}\frac{n_k}{k}\leq \frac{1}{L}<\infty.\qedhere\]
\end{proof}

\section{Gap Theorems}

While much of the classical theory of complex functions focuses on methods and conditions for analytic continuation, another aspect of complex function theory is concerned with what might be termed the opposite phenomenon: functions that possess no analytic continuation beyond the domain of validity of their initial representation. This boundary behavior is systematically characterized by gap theorems, which reveal a significant connection between the sparsity of a sequence of exponents and the emergence of a natural boundary. In this paper, we shall need the following gap theorems.

\begin{proposition}\begin{enumerate}\label{refProp}\fixitem\label{GrahamHare} \cite[Proposition~10.4.7]{gra13} If $(n_k)$ is a strictly increasing sequence of positive integers such that $\{n_k\  :\  k\in\N\}$ has zero harmonic density, then for any sequence $(a_k)$ of complex numbers, the function $f(z)=\sum_{k=1}^\infty a_kz^{n_k}$ cannot be analytically continued across the circle of convergence.
\item\label{Polya} \cite[p.~235]{seg08} If $(n_k)$ is a strictly increasing sequence of nonnegative integers such that $\kliminf\frac{n_k}{k}<\infty$, then there exists a sequence $(a_k)$ of complex numbers such that $f(z)=\sum_{k=1}^\infty a_kz^{n_k}$ has radius of convergence $1$ and may be analytically continued across $\T$.
\end{enumerate}
\end{proposition}

The proof of Proposition~\ref{refProp}\ref{GrahamHare} in \cite[p.~185]{gra13} relies on a localized rotation technique that leverages the measure-theoretic interpolation, that results from zero harmonic density, to propagate analyticity along the boundary of the disk of convergence of the power series. In particular, the Fourier coefficients of an interpolating measure were made to match the phase shifts of some rotation precisely along the exponent sequence. An auxiliary function, which is an integral of $f(z)$, was used to produce an analytic continuation of $f$ at some element of the boundary of the disk of convergence, and this was propagated by rotation around the unit circle, accomplishing a proof by contradiction. The literature on gap theorems in complex analysis is clear on what is the weakest condition for a natural boundary, which is in the statement of the Fabry Gap Theorem. The sufficient condition in this theorem is precisely that of the exponents in the power series having zero density, and this brings us to the other statement in Proposition~\ref{refProp}.

Although the exact statement of Proposition~\ref{refProp}\ref{Polya} above is something we took from \cite[p.~235]{seg08}, the theorem is actually a classical result by  P\'{o}lya \cite{pol42}. P\'{o}lya’s proof of the converse Fabry Gap Theorem relies on complex function theory, specifically utilizing the Borel transform and Phragm\'{e}n–Lindel\"{o}f indicator diagrams. P\'{o}lya constructed the power series coefficients $a_k$ by linking them to the Taylor coefficients of an auxiliary entire function of exponential type. By carefully selecting certain blocks of indices, that exhibit lacunarity, from the exponent sequence where the density is bounded, the coefficients were selected to correspond to an indicator diagram, a notion that P\'{o}lya developed in \cite{pol42}, the conjugate domain of which does not completely cover the unit circle. This analytical framework implicitly establishes that the resulting series possesses regular boundary points, and hence there is no more natural boundary for this case.

In contrast, Erd\H{o}s introduced a purely elementary, constructive approach that avoids the machinery of complex function theory. The coefficients $a_k$ were obtained by isolating specific blocks of intervals where the localized density of the exponent sequence remains bounded. Within these blocks, he framed the problem of analytic continuation as a system of homogeneous linear equations designed to force the coefficients of a shifted series, $f(z + \lambda)$, to vanish at targeted higher-order terms. By utilizing a counting argument to prove the number of available coefficients (as the variables of the system) strictly exceeds the number of vanishing constraints, Erd\H{o}s guaranteed a non-trivial solution. The coefficients $a_k$ are explicitly bounded to normalize the radius of convergence to exactly one while guaranteeing regularity at a boundary point via a direct translation argument. This was the approach used in the exposition in \cite[pp.~235--239]{seg08}.

Since the construction of the coefficients $a_k$ in Proposition~\ref{refProp}\ref{Polya} is based on at least two classical approaches that we have just cited, we refer the reader to the aforementioned works for the actual appearance of the $a_k$, and for this paper, invoking Proposition~\ref{refProp}\ref{Polya} as an existence result shall suffice. Towards achieving our main result, we now generalize Proposition~\ref{refProp}\ref{GrahamHare} into the following.

\begin{lemma}\label{zhdLem} 
If $E \sub \Z$ has zero harmonic density, then any Laurent series $f(z) = \sum_{\gamma \in E} a_\gamma z^{\gamma}$ cannot be analytically continued beyond or within its open annulus of convergence $A(r, R) := \{z \in \C : r < |z| < R\}$.
\end{lemma}

\begin{proof}
We decompose $f(z)$ as $f(z) = P(z) + N(z)$, where
$P(z) = \sum_{\gamma \in E^+} a_\gamma z^\gamma$ and $N(z) = \sum_{\gamma \in E^-} a_{-\gamma} z^{-\gamma}$. Let $R$ be the radius of convergence of $P(z)$, and let $R_N$ be the radius of convergence of $Q(\zeta) := \sum_{\gamma \in E^-} a_{-\gamma} \zeta^\gamma$. Defining $r := 1/R_N$, the series $N(z) = Q(1/z)$ converges whenever $|z| > r$. 

If $r \ge R$, then the open annulus of convergence $A(r,R)$ is empty, so no analytic continuation exists. We assume henceforth that $r < R$, so that $f(z)$ converges holomorphically on $A(r, R)$. 

By Proposition~\ref{subProp}\ref{zhdSub}, both $E^+$ and $E^-$ have zero harmonic density. Applying Proposition~\ref{refProp}\ref{GrahamHare} to $P(z)$, the circle $|z| = R$ is a natural boundary for $P(z)$. Similarly, applying Proposition~\ref{refProp}\ref{GrahamHare} to $Q(\zeta)$, the circle $|\zeta| = R_N$ is a natural boundary for $Q(\zeta)$, which implies that the circle $|z| = r$ is a natural boundary for $N(z) = Q(1/z)$. For the sum $P(z)+N(z)$, we simply take the intersection of the region $\{z\in\C\,:\,|z|<R\}$ for $P(z)$ and $\{z\in\C\,:\,|z|>r\}$ for $N(z)$. Consequently, $f(z) = P(z) + N(z)$ cannot be analytically continued across any point of the boundary circles $|z| = R$ or $|z| = r$.
\end{proof}

\section{Main result}

Sets of integers that have zero harmonic density are characterized by not containing arbitrarily long arithmetic progressions of fixed step length \cite[Proposition~10.2.8]{gra13}. Thus, the set $\{100^{j!}+k\,:\,k\leq j\}$, which may be routinely proven to have zero density, and which contains arbitrarily long arithmetic progressions of fixed step length, does not have zero harmonic density \cite[p.~186]{gra13}. Hence is the question of the connection between zero harmonic density and zero density for subsets of $\Z$, or alternatively, as phrased as an open problem in \cite[p.~249]{gra13}, ``When does a set (of integers) of zero density have zero harmonic density?'' Perhaps this question was phrased under the thought process that if $\zdSet$ is the collection of all subsets of $\Z$ that have zero density, and $\zhdSet$ as the collection of those with zero harmonic density, then the question is under what conditions will a set from $\zdSet$ be in $\zhdSet$. Our answer to the open problem is that $\zhdSet\setminus\zdSet=\emptyset$.

\begin{theorem}\label{TheThm} A set of integers that has zero harmonic density has zero density.
\end{theorem}
\begin{proof} Tending towards a contradiction, suppose there exists a set $E$ of integers with zero harmonic density that does not have zero density. By Proposition~\ref{subProp}\ref{zdSub}, there exists $S\in\{E^+,E^-\}$ such that $S$ does not have zero density. By Proposition~\ref{refProp}\ref{Polya}, there exists a sequence $(a_\gamma)_{\gamma\in S}$ of complex numbers such that either $\sum_{\gamma\in S=E^+}a_\gamma z^\gamma$ or $Q(\zeta):=\sum_{\gamma\in S=E^-}a_{-\gamma} \zeta^\gamma$ has radius of convergence $1$ and has an analytic continuation on a neighborhood of some $e^{i\alpha}\in\T$. If we define $a_\gamma:=2^{-|\gamma|}$ for all $\gamma\in E\setdiff S$, then $f(z)=\sum_{\gamma\in E}a_\gamma z^{\gamma}$ converges in the annular region $A(\frac{1}{2},1)$ and has an analytic continuation on some neighborhood of $e^{i\alpha}\in\T$ for the case $S=E^+$, or, for the case $S=E^-$, the series  $f(z)=Q(\frac{1}{z})+\sum_{\gamma\in E^+}a_\gamma z^{\gamma}$ converges in the annular region $A(1,2)$ and has an analytic continuation on some neighborhood of $e^{-i\alpha}\in\T$. In both cases, the analytic continuation at the aforementioned element of $\T$ is forbidden by Lemma~\ref{zhdLem}. Therefore, any set of integers with zero harmonic density has zero density.
\end{proof}

Given $U\sub\T$, the subset of $M(U)$ consisting of discrete measures is denoted by $M_d(U)$. Replacing $M(\T)$ and $M(U)$ by $M_d(\T)$ and $M_d(U)$, respectively, in the definition of zero harmonic density in Section~\ref{PrelimSec} results in the notion of \emph{zero discrete harmonic density}. The two important classes of interpolation sets covered in \cite{gra13}, which are $I_0$ sets and Sidon sets are sets with zero discrete harmonic density and with zero harmonic density, respectively. An $I_0$ set is a Sidon set, but not conversely. This is because the union of two $I_0$ sets may not be $I_0$, but the union of two Sidon sets is Sidon \cite[Section~6.3]{gra13}. (The class of $I_0$ sets includes Hadamard sets, which are the subject of the sufficient condition for the Hadamard Gap Theorem.) By \cite[Corollary~10.2.5]{gra13}, zero discrete harmonic density implies zero harmonic density, so by Theorem~\ref{TheThm}, all sets mentioned have zero density in $\Z$. Perhaps a good further direction is to investigate the zero density of these sets in relation to their density in the Bohr topology on $\Z$ \cite[Chapter~8]{gra13}, which is the topology on $\Z$ if its dual $\T$ is taken to have the discrete instead of the usual compact topology.

\section*{Ethical approval} This is not applicable to research work in pure mathematics.

\section*{Funding} The author was supported by the Research Grants Management Office of De La Salle University, Taft Ave., Manila, Philippines, with grant no.: 02FR1TAY24-1TAY25.

\section*{Availability of data and materials} Studies in pure mathematics do not involve data sets, and hence a declaration on availability of data and materials is not applicable.

\end{document}